\documentclass[11pt, leqno]{amsart}

\usepackage[mathscr]{eucal}
\usepackage{amsmath}
\usepackage{amsfonts,dsfont}
\usepackage{amsthm}
\usepackage{color}
\usepackage{placeins}

\theoremstyle{plain}
\newtheorem{theorem}{Theorem}[section]
\newtheorem{proposition}{Proposition}[section]
\newtheorem{lemma}{Lemma}[section]

\usepackage{a4wide}
\usepackage{amssymb}
\usepackage{graphicx}

\newtheorem{remark}{Remark}[section]

\numberwithin{equation}{section}

\dedicatory{Dedicated to professor Iskander Asanovich Taimanov on the occasion of his 60th birthday}

\title[Magnetic curves] 
{Homogeneity of magnetic trajectories \\
in the real special linear group}
\author[J. Inoguchi]{Jun-ichi Inoguchi}
\address[J. Inoguchi]
{Department of Mathematics,
Hokkaido University
Sapporo
060-0810 Japan}
\email{inoguchi@math.sci.hokudai.ac.jp}

\author[M.~I.~Munteanu]{Marian Ioan Munteanu}
\address[M.~I.~Munteanu]
{University 'Al. I. Cuza' of Iasi, 
Faculty of Mathematics, Bd. Carol I, no.~11,
700506 Iasi, Romania}
\email{marian.ioan.munteanu@gmail.com}

\date{{\color{blue}{December 18, 2023 (published electronically)\\
Proc. Amer. Math. Soc.\qquad
DOI: https://doi.org/10.1090/proc/16596 }} (an abridged version)}

\begin{document}

\begin{abstract}
We prove the  homogeneity of contact magnetic curves in the real 
special linear group of degree $2$. Every 
contact magnetic trajectory is a product of a
homogeneous geodesic and a 
charged Reeb flow.
\end{abstract}



\maketitle

\section*{Introduction}
According to Thurston's  
\emph{Geometrization Conjecture} (proved by Perelman), 
any compact orientable $3$-fold can be cut in a special way into pieces 
admitting one of the eight geometric structures: 
the Euclidean, spherical, and hyperbolic structures 
($\mathbb{E}^3$, $\mathbb{S}^3$, $\mathbb{H}^3$), 
two product manifolds ($\mathbb{S}^{2}\times\mathbb{R}$ and 
$\mathbb{H}^{2}\times\mathbb{R}$), 
and the $3$-dimensional Lie groups ($\mathrm{Nil}_3$, $\mathrm{Sol}_3$ and 
$\widetilde{\mathrm{SL}}_2\mathbb{R}$). Each model spaces admits compatible 
Riemannian metric. The resulting Riemannian $3$-manifolds 
are naturally reductive homogeneous spaces except $\mathrm{Sol}_3$. 
Notice that compact Riemannian symmetric spaces are normal homogeneous spaces. 
It should be remarked that the class of naturally reductive homogeneous spaces 
is strictly wider than that of normal homogeneous space. Indeed, among the eight 
model spaces of Thurston geometry, the only normal homogeneous spaces are spheres.
The class of simply connected $3$-dimensional 
naturally reductive homogeneous space is exhausted by 
the above seven model spaces together with Berger $3$-spheres.

Dynamical systems 
on $3$-manifolds, especially, the model spaces (and their compact 
quotients) of Thurston geometry have been 
paid much attention of mathematicians. 
An important study concerning the relation between the integrability and the
vanishing of the topological entropy of the geodesic flow is done by 
Bolsinov and Taimanov. In \cite{BT}, they gave a 
$\mathbb{T}^2$-bundle over a circle $\mathbb{S}^1$ which provides 
an example of Liouville integrable geodesic flow with positive topological entropy 
and vanishing Liouville entropy.

We concentrate our attention to $\widetilde{\mathrm{SL}}_{2}\mathbb{R}$
and its quotients. The special linear group 
$\mathrm{SL}_2\mathbb{R}$ is a \emph{non-normal} naturally 
reductive homogeneous space. Thus all of geodesics 
are \emph{homogeneous}. This fact implies that 
geodesic flows of $\mathrm{SL}_2\mathbb{R}$ are 
\emph{integrable} in non-commutative sense 
(see \cite{BJ03,Jov}). 
According to \cite{BVY}, the fact that the modular 
$3$-fold $\mathrm{PSL}_{2}\mathbb{R}/\mathrm{PSL}_{2}\mathbb{Z}$ 
is topologically equivalent to the complement of the 
trefoil $\mathcal{K}$ in the $3$-sphere was first observed by Quillen 
(see Milnor \cite{Milnor}). 
Bolsinov, Veselov and Ye \cite{BVY} proved that
the periodic geodesics on the modular $3$-fold 
$\mathrm{SL}_{2}\mathbb{R}/\mathrm{SL}_{2}\mathbb{Z}$ 
with sufficiently large values of $\mathcal{C}=\kappa^2/16$
represent trefoil cable knots in 
$\mathbb{S}^3\smallsetminus \mathcal{K}$, where $\kappa$ is the 
geodesic curvature of the projected curve in the hyperbolic surface. 
Any trefoil cable knot can be described in this way.
These facts really prove that geodesics flows in $\mathrm{SL}_2\mathbb{R}$ 
and its compact quotients are very attractive dynamical system from both 
differential geometry and low-dimensional topology perspectives.

It would be interesting to perturbe geodesic flows of 
$\mathrm{SL}_2\mathbb{R}$. Concerning this 
viewpoint, Arnol'd \cite{Arnold61} suggested to 
perturbe the symplectic  form of the 
Hamiltonian system of geodesic flow of arbitrary Riemannian 
manifolds by a closed $2$-form on the configuration manifold. 
The $2$-form is regarded as a \emph{static magnetic field} 
of the configuration manifold. Since then 
magnetic trajectories on Riemannian manifolds 
have been paid much attention of mathematicians. 
In particular, analysis and  classification of periodic 
trajectories have been studied extensively, see 
\textit{e.g.}, Bahri and Taimanov \cite{BaTa}. 
On the special linear group 
$\mathrm{SL}_2\mathbb{R}$, there exists a 
unit Killing vector field $\xi$ (Reeb vector field) which is identified 
with the magnetic field (called the 
\emph{contact magnetic field}) whose potential is 
the canonical contact structure. 
The differential equation of the magnetic 
trajectory strongly reflects 
geometric properties of $\mathrm{SL}_2\mathbb{R}$. 
In this sense the system of contact magnetic trajectory is a 
nice perturbation of geodesic flow. 
In our previous work \cite{IM1} we investigated 
periodicity of contact magnetic trajectories in 
$\mathrm{SL}_2\mathbb{R}$.

In \cite{BJ}, 
Bolsinov and Jovanovi{\'c} studied 
magnetic trajectories in normal homogeneous 
Riemannian spaces. They showed that 
every magnetic trajectory 
starting at the origin 
is homogeneous (see \cite[Remark 1]{BJ}). 
We expect the homogeneity of magnetic trajectories
on more wider class of homogeneous Riemannian spaces. 
As our first attempt, in this article we prove that 
magnetic trajectories in the special linear group $\mathrm{SL}_2\mathbb{R}$ 
(derived from the canonical contact structure) are homogeneous.

More precisely, we show that every contact magnetic trajectory 
$\gamma$ of charge $q$ starting at the origin 
$\mathrm{Id}$ of $\mathrm{SL}_2\mathbb{R}$ with initial velocity $X$ 
and with charge $q$ is the product of the homogeneous geodesic 
$\gamma_{X}(s)$ with initial velocity $X$ and the 
\emph{charged Reeb flow} $\exp_{\mathrm{SO}(2)}\{s(2q\xi)\}$.

\section{Preliminaries}
\subsection{Homogeneous geometry}
Let $N=\mathcal{G}/\mathcal{H}$ be a homogeneous space. 
Denote by $\mathfrak{G}$ and $\mathfrak{H}$ the Lie algebras 
of $\mathcal{G}$ and $\mathcal{H}$, respectively. 
Then $N$ is said to be \emph{reductive} if there exits a 
linear subspace $\mathfrak{p}$ of $\mathfrak{G}$ complementary 
to $\mathfrak{H}$ and satisfies 
$[\mathfrak{H},\mathfrak{p}]\subset\mathfrak{p}$. It is known that 
every homogeneous Riemannian space is reductive. 

Now let $N=\mathcal{G}/\mathcal{H}$ be a homogeneous Riemannian space with 
reductive decomposition $\mathfrak{G}=\mathfrak{H}+\mathfrak{p}$ and a 
$\mathcal{G}$-invariant Riemannian metric 
$g=\langle\cdot,\cdot\rangle$. Then $N$ is said to be \emph{naturally reductive} 
(with respect to $\mathfrak{p}$) if it satisfies
\begin{equation}\label{eq:NatRed}
\langle [X,Y]_{\mathfrak{p}},Z\rangle
+\langle Y, [X,Z]_{\mathfrak{p}}\rangle=0
\end{equation}
for any $X$, $Y$, $Z\in\mathfrak{p}$. Here 
we denote the $\mathfrak{p}$ part of a vector 
$X\in\mathfrak{G}$ by $X_{\mathfrak{p}}$. 

A homogeneous Riemannian space 
$N=\mathcal{G}/\mathcal{H}$ is said to be 
\emph{normal} if $G$ is compact semi-simple and 
the metric is derived from the restriction 
of the bi-invariant Riemannian metric of $G$. 
Normal homogeneous spaces are naturally reductive.

As a generalization of naturally reductive homogeneous 
space, the notion of Riemannian g.~o.~space 
was introduced by Kowalski and Vanhecke \cite{KV}. 

According to \cite{KV}, a homogeneous Riemannian manifold 
$N=\mathcal{G}/\mathcal{H}$ is called a 
\emph{space with homogeneous geodesics} or a 
\emph{Riemannian g.o.~space} if every geodesic 
$\gamma(s)$ of $M$ is an orbit of a one-parameter subgroup of 
the \emph{largest} connected group of isometries. 
Naturally reductive homogenous spaces are typical examples of 
Riemannian g.o.~spaces. (For more informations, we refer to \cite{Av}).

\subsection{The Euler-Arnold equation}
Let $G$ be a Lie group equipped with
a left invariant Riemannian metric $\langle\cdot,\cdot\rangle$. Denote 
by $\mathfrak{g}$ the Lie algebra of $G$. 
The bi-invariance obstruction ${\mathsf{U}}$ is a 
symmetric bilinear map 
$\mathsf{U}:\mathfrak{g}\times\mathfrak{g}\to\mathfrak{g}$ 
defined by
\begin{equation}\label{eq:U}
2\langle {\mathsf{U}}(X,Y),Z \rangle=
-\langle X,[Y,Z]\rangle+\langle Y,[Z,X] \rangle,
\ X,Y \in {\frak g}
\end{equation}
The Levi-Civita connection $\nabla$ is described as
\[
\nabla_{X}Y=\frac{1}{2}[X,Y]+\mathsf{U}(X,Y),
\quad X,Y\in\mathfrak{g}.
\]
On the Lie algebra $\mathfrak{g}$, we define 
the linear operator $\mathrm{ad}^{*}:\mathfrak{g}
\to\mathfrak{gl}(\mathfrak{g})$ by
\[
\langle \mathrm{ad}(X)Y,Z\rangle
=\langle Y,\mathrm{ad}^{*}(X)Z\rangle,
\quad X,Y,Z\in\mathfrak{g}.
\]
One can see that 
\[
-2\mathsf{U}(X,Y)=\mathrm{ad}^{*}(X)Y+\mathrm{ad}^{*}(Y)X.
\]
Take a curve $\gamma(s)$ starting at the origin (identity $\mathrm{Id}$) of $G$. 
Set $\varOmega(s)=\gamma(s)^{-1}\dot{\gamma}(s)$, then 
one can check that 
\begin{equation}\label{eq:1.1}
\nabla_{\dot{\gamma}}\dot{\gamma}
=\gamma\left(
\dot{\varOmega}-\mathrm{ad}^{*}(\varOmega)\varOmega
\right).
\end{equation}
This implies the 
following fundamental fact \cite[Appendix 2]{Arnold}:
\begin{proposition}
The curve $\gamma(s)$ is a geodesic if and only if
\[
\dot{\varOmega}=\mathrm{ad}^{*}(\varOmega)\varOmega.
\]
\end{proposition}
Now let us assume that $G$ is semi-simple. Then 
$G$ admits a bi-invariant semi-Riemannian metric  
$\mathsf{B}$. It is a constant multiple of the 
Killing form. On this reason we call the metric by the name 
(normalized) \emph{Killing metric} and denote 
it by $\mathsf{B}$. Introduce an endomorphism field 
$\mathcal{I}$ (called the \emph{moment of inertia tensor field}) on $\mathfrak{g}$ by
\[
\langle X,Y\rangle=\mathsf{B}(\mathcal{I}X,Y),
\quad \mathcal{I}X,Y\in\mathfrak{g}
\]
and the (\emph{anglar}) \emph{momentum} (see \cite[Appendix 2.C]{Arnold}) by
\[
\mu=\mathcal{I}\varOmega.
\]
Then $\dot{\varOmega}=\mathcal{I}^{-1}\dot{\mu}$ and hence 
\[
\langle\nabla_{\dot{\gamma}}\dot{\gamma},\gamma Z
\rangle
=\langle \mathcal{I}^{-1}(\dot{\mu}-[\mu,\varOmega]),Z
\rangle.
\]
Thus we arrive at the so-called 
\emph{Euler-Arnold equation} \cite{MPMM} (also called the 
\emph{Euler-Poincar{\'e} equation} \cite{HMR})
\begin{equation}\label{eq:EA}
\dot{\mu}-[\mu,\varOmega]=0.
\end{equation}
When $G=\mathrm{SO}(3)$ equipped with Killing metric, 
$\varOmega_{\mathsf{c}}:=\gamma(t)^{-1}\dot{\gamma}(t)$ and 
$\varOmega_{\mathsf{s}}:=\dot{\gamma}(t)\gamma(t)^{-1}$ are 
\emph{anglar velocity in the body} (\emph{corpus}) and 
\emph{angular velocity in the space}, respectively. 
The terminologies are due to Arnold and are justified by the fact that
the motion of a rigid body in Euclidean $3$-space under inertia is a geodesic 
in $\mathrm{SO}(3)$ equipped with Killing metric.

\subsection{Static magnetism}
Let $(M,g)$ be a Riemannian manifold. 
A (static) \emph{magnetic field} is a closed 
$2$-form $F$ on $M$. The \emph{Lorentz force} $\varphi$ derived from 
$F$ is an endomorphim field defined by
\[
F(X,Y)=g(\varphi X,Y),
\quad X,Y\in\varGamma(TM).
\]
A curve $\gamma(s)$ is said to be a 
\emph{magnetic trajectory} under the influence 
of $F$ if it obeys the 
\emph{Lorentz equation}
\[
\nabla_{\dot{\gamma}}\dot{\gamma}=q\varphi\dot{\gamma},
\] 
where $\nabla$ is the Levi-Civita connection of $(M,g)$, 
$\dot{\gamma}$ is the velocity of $\gamma(s)$ and 
$q$ is a constant called the \emph{charge}.

The Lorentz equation implies that 
magnetic trajectories are of constant speed.
Note that when $F=0$ or $q=0$, magnetic trajectories 
reduce to geodesics. For more information on 
magnetic trajectories, we refer to \cite{IM2}.

\section{The special linear group} 
\subsection{Iwasawa decomposition}
Let 
$\mathrm{SL}_{2}\mathbb{R}$ be
the real special linear group of degree $2$:
\[
\mathrm{SL}_{2}\mathbb{R}=\left \{
\left (
\begin{array}{cc}
a & b \\ c & d
\end{array}
\right )
\ \biggr \vert \
a,b,c,d \in \mathbb{R},
\ ad-bc=1
\right \}.
\]
The Iwasawa decomposition 
$\mathrm{SL}_2\mathbb{R}=NAK$ of $\mathrm{SL}_2\mathbb{R}$; 
\begin{align*}
N&=\left \{
\left (
\begin{array}{cc}
1 & x \\ 0 & 1 
\end{array}
\right )
\
\biggr \vert \ 
x \in \mathbb{R}
\right \}, \qquad
\mbox{(Nilpotent \ part)}
\\
A&=
\left \{
\left (
\begin{array}{cc}
\sqrt{y} & 0 \\ 0 & 1/\sqrt{y}
\end{array}
\right )
\ \biggr
\vert \
y>0
\right \}, \qquad
\mbox{(Abelian \ part)}
\\
K&=
\left \{
\left (
\begin{array}{cc}
\cos \theta & \sin \theta \\
-\sin \theta & \cos \theta 
\end{array}
\right )
\ \biggr \vert 
\ 0\leq \theta < 2\pi
\right \}=\mathrm{SO}(2), \quad
\mbox{(Maximal \ torus)}
\end{align*}
allows to introduce the following
global coordinate system $(x,y,\theta)$
of $\mathrm{SL}_2\mathbb{R}$:
\begin{equation} \label{coord}
(x,y,\theta)\longmapsto
\left (
\begin{array}{cc}
1 & x \\ 0 & 1 
\end{array}
\right )
\left (
\begin{array}{cc}
\sqrt{y} & 0 \\ 0 & 1/\sqrt{y}
\end{array}
\right )
\left (
\begin{array}{cc}
\cos \theta & \sin \theta \\
-\sin \theta & \cos \theta 
\end{array}
\right ).
\end{equation}
The mapping 
\[
\psi:\mathbb{H}^2(-4) \times 
\mathbb{S}^1 \rightarrow \mathrm{SL}_2\mathbb{R};\ \ 
\psi(x,y,\theta):=
\begin{pmatrix} 1 & x \\ 0 & 1 
\end{pmatrix}
\begin{pmatrix}
\sqrt{y} & 0 \\ 0 & 1/\sqrt{y} 
\end{pmatrix} 
\begin{pmatrix}
\cos \theta & \sin \theta \\ 
-\sin \theta & \cos \theta
\end{pmatrix}
\]
is a diffeomorphism onto 
$\mathrm{SL}_2\mathbb{R}$. 
Hereafter, we shall refer $(x,y,\theta)$ as a global
coordinate system of $\mathrm{SL}_2\mathbb{R}$.
Hence $\mathrm{SL}_2\mathbb{R}$ is diffeomorphic to
$\mathbb{R} \times \mathbb{R}^{+}\times\mathbb{S}^1$
and hence diffeomorphic to $\mathbb{R}^{3}\setminus \mathbb{R}$.
Since $\mathbb{R} \times \mathbb{R}^{+}$ is diffeomorphic
to open unit disk $\mathbb{D}$, then $\mathrm{SL}_2\mathbb{R}$ is diffeomorphic
to open solid torus $\mathbb{D} \times \mathbb{S}^1$.

The Iwasawa decomposition of 
$\mathrm{SL}_2\mathbb{R}$ can be carried out 
explicitly:

\begin{proposition}\label{Iwasawa}
The Iwasawa decomposition of an 
element $p=\left(\begin{array}{cc}
p_{11}&p_{12}\\p_{21}&p_{22}
\end{array}\right)
\in\mathrm{SL}_{2}\mathbb{R}$ is given 
explicitly by $p=n(p)a(p)k(p)$, where 
\[
n(p)=\left(
\begin{array}{cc}
1 & x\\
0 &1
\end{array}
\right), \ \ 
a(p)=\left(
\begin{array}{cc}
\sqrt{y} & 0\\
0 &1/\sqrt{y}
\end{array}
\right), \ \ 
k(p)=\left(
\begin{array}{cc}
\cos\theta & \sin\theta\\
-\sin\theta &\cos\theta
\end{array}
\right)
\]
with 
\[
x=\frac{p_{11}p_{21}+p_{12}p_{22}}
{(p_{21})^2+(p_{22})^2},
\quad 
y=\frac{1}{(p_{21})^2+(p_{22})^2},
\quad 
e^{i\theta}=\frac{p_{22}-ip_{21}}{\sqrt{(p_{21})^2+(p_{22})^2}}.
\]
\end{proposition}

\subsection{The standard Riemannian metric}
The Lie algebra $\mathfrak{sl}_2\mathbb{R}$ 
of $\mathrm{SL}_2\mathbb{R}$ is given explicitly by
\[
\mathfrak{sl}_2\mathbb{R}
=\left\{
X \in \mathrm{M}_{2}\mathbb{R}\ \biggr \vert
\ \mathrm{tr}\> X=0 \right \}. 
\]
We take the following basis of $\mathfrak{sl}_2\mathbb{R}$:
\[
E=\begin{pmatrix} 0 & 1 \\ 0 & 0 \end{pmatrix},\quad
F=\begin{pmatrix} 0 & 0 \\ 1 & 0 \end{pmatrix},\quad
H=\begin{pmatrix} 1 & 0 \\ 0 & -1 \end{pmatrix}.
\]
This basis satisfies the commutation relations:
\[
[E,F]=H,\quad 
[F,H]=2F,\quad 
[H,E]=2E.
\]
The Lie algebra $\mathfrak{n}$, 
$\mathfrak{a}$ and $\mathfrak{k}$ 
of the closed subgroups $N$, $A$ and $K$ are given by 
\[
\mathfrak{n}=\mathbb{R}E,\quad
\mathfrak{a}=\mathbb{R}H,\quad
\mathfrak{k}=\mathbb{R}(E-F).
\]
The Lie algebra $\mathfrak{h}$ is the 
Cartan subalgebra of $\mathfrak{sl}_{2}\mathbb{R}$. 
Moreover $\mathfrak{n}$ and $\mathbb{R}F$ are 
root spaces with respect to $\mathfrak{h}$.
The decomposition 
$\mathfrak{sl}_{2}\mathbb{R}=\mathfrak{h}
\oplus \mathfrak{n}\oplus\mathbb{R}F$ is the root 
space decomposition (or Gauss decompostion) of 
$\mathfrak{sl}_{2}\mathbb{R}$.

Hereafter we use the left invariant 
frame field $\{e_1=E-F,e_2=E+F, e_3=H\}$. 
These left invariant vector fields are 
given explicitly by
\begin{align*}
e_1=&\frac{\partial}{\partial \theta},\\
e_2=&\cos(2\theta)
\left(
2y\frac{\partial}{\partial x}
-\frac{\partial}{\partial \theta}
\right)
+\sin(2\theta)\left(
2y\frac{\partial}{\partial y}
\right),
\\
e_3=&-
\sin(2\theta)
\left(
2y\frac{\partial}{\partial x}
-\frac{\partial}{\partial \theta}
\right)+\cos(2\theta)\left(
2y\frac{\partial}{\partial y}
\right).
\end{align*}
Define an inner product 
$\langle \cdot, \cdot\rangle$
so that 
$\{e_1,e_2,e_3\}$
is orthonormal with respect to 
$\langle \cdot,\cdot \rangle$.
By left-translating this inner product,
we equip a left invariant Riemannian metric
\[
g=
\frac{dx^2+dy^2}{4y^2}+
\left(d\theta+\frac{dx}{2y}
\right)^2.
\]
The one form 
\[
\eta=d\theta+\frac{dx}{2y}
\]
is a globally defined 
contact form on $\mathrm{SL}_2\mathbb{R}$. 
The Reeb vector field $\xi$ of the 
contact form $\eta$ is $\xi=e_1$.

The universal covering 
space of $(\mathrm{SL}_2\mathbb{R},g)$ 
is one of the model space of Thurston geometry 
\cite{Thurston}. 

\subsection{The hyperbolic Hopf fibering}
The special linear group $\mathrm{SL}_2\mathbb{R}$ 
acts transitively and
isometrically on the upper half plane:
\[
\mathbb{H}^2(-4)=
\left(
\{(x,y)\in \mathbb{R}^2 \ | \ y>0\},
\frac{dx^2+dy^2}{4y^2}
\right)
\]
of constant curvature $-4$ by the 
linear fractional transformation as
\[
\left(
\begin{array}{cc}
a & b\\
c & d
\end{array}
\right)\cdot z=\frac{az+b}{cz+d}.
\]
Here we regard a point $(x,y)\in\mathbb{H}^2(-4)$ as a complex 
number $z=x+yi$.


The isotropy subgroup of $\mathrm{SL}_2\mathbb{R}$ at $i=(0,1)$ is 
the rotation group $\mathrm{SO}(2)$.
The natural projection $\pi:(\mathrm{SL}_2\mathbb{R},g)
\to \mathrm{SL}_2\mathbb{R}/\mathrm{SO}(2)=\mathbb{H}^2(-4)$
is given explicitly by
\[
\pi(x,y,\theta)=(x,y) \in \mathbb{H}^2(-4)
\]
in terms of the global coordinate
system \eqref{coord}.

The tangent space $T_{i}\mathbb{H}^{2}(-4)$ at 
the origin $i=(0,1)$ is identified with the 
vector subspace $\mathfrak{m}$ defined by
\[
\mathfrak{m}=\{X\in\mathfrak{sl}_{2}\mathbb{R}\>|\>
{}^t\!X=X\}.
\]
The Lie algebra $\mathfrak{g}=\mathfrak{sl}_2\mathbb{R}$ has 
the orthogonal splitting $\mathfrak{g}=\mathfrak{k}
\oplus\mathfrak{m}$.
This splitting can be carried out explicitly as
\[
X=X_{\mathfrak k}+X_{\mathfrak m}, \ \ 
X_{\mathfrak k}=\frac{1}{2}(X-{}^t\!X),
\ \ 
X_{\mathfrak m}=\frac{1}{2}(X+{}^t\!X).
\]
It is easy to see that the projection 
$\pi:(\mathrm{SL}_2\mathbb{R},g)
\rightarrow \mathrm{SL}_2\mathbb{R}/\mathrm{SO}(2)
=\mathbb{H}^2(-4)$ is a Riemannian submersion 
with totally geodesic fibres. 
This submersion 
$\pi:(\mathrm{SL}_2\mathbb{R},g) \to \mathbb{H}^2(-4)$ 
is called the
\emph{hyperbolic Hopf fibering} of $\mathbb{H}^2(-4)$. 

Under the identification 
$\mathfrak{k}\cong\mathbb{R}$, 
the contact form $\eta$ is regarded as a 
connection form of the principal 
circle bundle $\mathrm{SL}_2\mathbb{R}
\to\mathbb{H}^2(-4)$.

\subsection{The naturally reductive structure}
On the Lie algebra 
$\mathfrak{g}=\mathfrak{sl}_2\mathbb{R}$, the inner product 
$\langle\cdot,\cdot\rangle$ at the identity 
induced from 
$g$ is written as
\[
\langle X,Y \rangle=
\frac{1}{2} \ \mathrm{tr}\>
({}^{t}\!XY ),\ \ 
X,Y \in \mathfrak{sl}_{2}\mathbb{R}.
\]
One can see that the metric $g$
is not only invariant by 
$\mathrm{SL}_2\mathbb{R}$-left
translation
but also right translations by 
$\mathrm{SO}(2)$.
Hence the Lie group $\mathcal{G}=\mathrm{SL}_2\mathbb{R}
\times\mathrm{SO}(2)$ 
with multiplication:
\[
(a,b)(a^{\prime},b^{\prime})=(aa^{\prime},bb^{\prime})
\]
acts isometrically on $\mathrm{SL}_2\mathbb{R}$ via the 
action:
\[
(\mathrm{SL}_2\mathbb{R}
\times\mathrm{SO}(2))\times\mathrm{SL}_2\mathbb{R}
\rightarrow \mathrm{SL}_2\mathbb{R};\ \ \ 
(a,b) \cdot X=aXb^{-1}.
\] 
Furthermore, this action of $\mathrm{SL}_2\mathbb{R}
\times\mathrm{SO}(2)$ on $\mathrm{SL}_2\mathbb{R}$ is transitive,
hence $\mathrm{SL}_2\mathbb{R}$ is a homogeneous Riemannian space
of $\mathrm{SL}_2\mathbb{R} \times\mathrm{SO}(2)$. 
The isotropy subgroup $\mathcal{H}$ of 
$\mathrm{SL}_2\mathbb{R}
\times\mathrm{SO}(2)$ at the identity matrix  $\mathrm{Id}$ is 
the diagonal subgroup
\[
\Delta K=\{(k,k)\>\vert\> k\in K\}\cong K
\]
of $K\times K$. 
The coset space $(\mathrm{SL}_2\mathbb{R}
\times\mathrm{SO}(2))/\mathrm{SO}(2)$
is a reductive homogeneous
space. The Lie algebra $\mathfrak{G}$ of the product group 
$\mathcal{G}=G\times K$ is $\mathfrak{G}=\mathfrak{g}\oplus\mathfrak{k}$.
On the other hand the Lie algebra $\mathfrak{H}$ of $\mathcal{H}=\Delta K$ is
\[
\Delta\mathfrak{k}=\{(W,W)\>\vert\>W\in\mathfrak{k}\}\cong 
\mathfrak{k}.
\]
The tangent space $T_{\mathrm{Id}}\mathrm{SL}_2\mathbb{R}$ 
of 
$\mathcal{G}/\mathcal{H}=
(\mathrm{SL}_2\mathbb{R}\times\mathrm{SO}(2))/
\Delta K$ is the 
Lie algebra $\mathfrak{g}=\mathfrak{sl}_2\mathbb{R}$. This tangent space is 
identified with the vector subspace $\mathfrak{p}$ 
of $\mathfrak{G}=\mathfrak{g}\times\mathfrak{k}$ defined by
\[
\mathfrak{p}=\{(V+W,2W)\>\vert\>V\in\mathfrak{m},
\>
W\in\mathfrak{k}\}.
\]
The Lie algebra $\mathfrak{g}\oplus
\mathfrak{k}$ is 
decomposed as
$\mathfrak{g}\oplus\mathfrak{k}=\Delta\mathfrak{k}
\oplus\mathfrak{p}$. One can see that 
this decomposition is reductive. Every $(X,Y)\in\mathfrak{g}\oplus
\mathfrak{k}$ is decomposed as
\[
(X,Y)=(2X_{\mathfrak k}-Y,2X_{\mathfrak k}-Y)
+(X_{\mathfrak m}+(Y-X_{\mathfrak k}),
2(Y-X_{\mathfrak k})).
\]
One can see that the linear subspace $\mathfrak{p}$ satisfies 
\eqref{eq:NatRed}. Thus $(\mathrm{SL}_{2}\mathbb{R}\times\mathrm{SO}(2))/\mathrm{SO}(2)$
is naturally reductive with respect to the 
decomposition $\mathfrak{g}\oplus\mathfrak{k}=\Delta\mathfrak{k}
\oplus\mathfrak{p}$. 
It is known as the only naturally reductive 
space representation of $\mathrm{SL}_2\mathbb{R}$ up to 
isomorphism (see \cite{HI06}).

\subsection{Curvatures}
The commutation relations of $\{e_1,e_2,e_3\}$ are 
\[
[e_1,e_2]=2e_3,\quad [e_2,e_3]=-2e_1,\quad  [e_3,e_1]=2e_2.
\]
The Levi-Civita connection $\nabla$ of
is given by
\[
\begin{array}{ccc}
\nabla_{e_1}e_{1}=0, & \nabla_{e_1}e_{2}=3e_{3}, & \nabla_{e_1}e_{3}=-3e_{2}\\
\nabla_{e_2}e_{1}=e_{3}, & \nabla_{e_2}e_{2}=0, & \nabla_{e_2}e_{3}=-e_{1}\\
\nabla_{e_3}e_{1}=-e_2, & \nabla_{e_3}e_{2}=e_{1} & \nabla_{e_3}e_{3}=0.
\end{array}
\]
The Riemannian curvature $R$ defined by
\[
R(X,Y):=\nabla_{X}\nabla_{Y}-
\nabla_{Y}\nabla_{X}-
\nabla_{[X,Y]}
\]
is given by
\[
R(e_1,e_2)e_1=-e_{2},\quad 
R(e_1,e_2)e_2=e_{1},
\]
\[
R(e_2,e_3)e_2=7e_{3},\quad 
R(e_2,e_3)e_3=-7e_{2},
\]
\[
R(e_1,e_3)e_1=-e_{3},\quad 
R(e_1,e_3)e_3=e_{1}. 
\]
The basis $\{e_1,e_2,e_3\}$ diagonalizes the Ricci tensor.
The principal Ricci curvatures are given by
\[
\rho_{1}=2,\quad 
\rho_{2}=\rho_{3}=-6.
\]
The bi-invariance obstruction $\mathsf{U}$ 
defined by
\[
2\langle 
\mathsf{U}(X,Y),Z\rangle
=-\langle X,[Y,Z]\rangle
+\langle Y,[Z,X]\rangle,
\quad 
X,Y,Z\in\mathfrak{sl}_{2}\mathbb{R}
\]
is given by
\begin{equation}\label{U-tensor}
\mathsf{U}(e_1,e_2)=2e_3,\quad
\mathsf{U}(e_1,e_3)=-2e_2.
\end{equation}
All the other components are zero.

From these we obtain
\[
\mathsf{U}(X,Y)=[X_{\mathfrak k},Y_{\mathfrak m}]+[Y_{\mathfrak k},X_{\mathfrak m}],
\quad  X,Y 
\in \mathfrak{g}.
\]
The Levi-Civita connection is 
rewritten as 
\begin{equation}\label{LC}
\nabla_{X}Y=\frac{1}{2}[X,Y]+{\mathsf{U}}(X,Y)
=\frac{1}{2}[X,Y]+[X_{\mathfrak k},Y_{\mathfrak m}]+[Y_{\mathfrak k},X_{\mathfrak m}],
\quad   X,Y 
\in \mathfrak{g}.
\end{equation}

\subsection{The contact magnetic field}
The contact form $\eta$ gives a left invariant magnetic 
field $F=d\eta$ called the \emph{contact magnetic field}
(see \cite{CFG,I2004,IM1,IM2}). 
The Lorentz force $\varphi$ of $F$
is given by
\[
\varphi e_1=0,\quad \varphi e_2=e_3, \quad  \varphi e_3=-e_2.
\]
Then the quartet $(\varphi,\xi,\eta,g)$ satisfies the 
following relations:
\[
\varphi^2=-I+\eta \otimes \xi,
\quad \nabla_{X}\xi=\varphi{X},
\]
\[
g(\varphi{X},\varphi{Y})=g(X,Y)-\eta(X)\eta(Y),
\]
\[
(\nabla_{X}\varphi)Y=-g(X,Y)\xi+\eta(Y)X,
\]
for all $X$, $Y \in \varGamma(T\mathrm{SL}_2\mathbb{R})$. 
These formulas imply that the
structure $(\varphi,\xi,\eta)$ is a 
left invariant Sasakian structure of $\mathrm{SL}_2\mathbb{R}$.
The structure $(\varphi,\xi,\eta)$ is called 
the \emph{canonical Sasakian structure} of 
$\mathrm{SL}_2\mathbb{R}$. The Riemannian curvature tensor $R$ of the metric $g$ is 
described by the
following formulas:
\begin{align*}
R(X,Y)Z  = & -
g(Y,Z)X+g(Z,X)Y \\
&  -2\>\{
\eta(Z)\eta(X)Y
-\eta(Y)\eta(Z)X \\
&+g(Z,X)\eta(Y)\xi
-g(Y,Z)\eta(X)\xi \\
	&-g(Y,\varphi{Z})\varphi{X}-g(Z,\varphi{X})\varphi{Y}+
2g(X,\varphi{Y})\varphi{Z}
\>\}
\end{align*}
in terms of the canonical Sasakian structure.

\section{The Euler-Arnold equation of $\mathrm{SL}_2\mathbb{R}$}
\subsection{}
In this section we deduce the 
Euler-Arnold equation for $\mathrm{SL}_2\mathbb{R}$. 
First of all we recall the notion of contact angle. 
Let $\gamma(s)$ be an arc length parametrized curve, then 
its \emph{contact angle} $\sigma(s)$ is the angle function 
between Reeb vector field and $\dot{\gamma}(s)$, 
\textit{i.e.}, $\cos\sigma(s)=g(\dot{\gamma}(s),\xi)$. 
An arc length parametrized cure is said to be a 
\emph{slant curve} if its contact angle is constant. 
In particular, $\gamma(s)$ is said to be a 
\emph{Reeb flow} if $\sin\sigma(s)=0$ along $\gamma(s)$. 
An arc length parametrized curve is called a \emph{Legendre curve} if $\cos\sigma(s)=0$.
One can see that every geodesic is a slant curve.
\subsection{}
The Killing form 
of $\mathrm{SL}_2\mathbb{R}$ is given by
\[
4\mathrm{tr}(XY),\quad X,Y\in\mathfrak{sl}_2\mathbb{R}.
\]
In this article we use the normalized Killing metric 
defined by
\[
\mathsf{B}(X,Y)=\frac{1}{2}\mathrm{tr}(XY),\quad X,Y\in\mathfrak{sl}_2\mathbb{R}.
\] 
As is well known the normalized Killing metric is a bi-invariant 
Lorentzian metric of constant curvature $-1$. Thus 
$(\mathrm{SL}_2\mathbb{R},\mathsf{B})$ is identified 
with the anti de Sitter $3$-space $\mathbb{H}^3_1$. 
Here we explain the reason why we choose the normalization as above. 
The naturally reductive Riemannian metric $g$ is directly connected to 
the anti de Sitter metric. Indeed, the anti de Sitter metric 
is given by 
\[
g-2\eta\otimes\eta=
\frac{dx^2+dy^2}{4y^2}-
\left(d\theta+\frac{dx}{2y}
\right)^2.
\]
Moreover the 
endomorphism field $\mathcal{I}$ is given by a simple form
\[
\mathcal{I}(X)={}^t\!X,\quad X\in\mathfrak{g}.
\]
Hence the Euler-Arnold equation is rewritten as 
\begin{equation}\label{eq:EA-SL2}
\dot{\mu}=[\mu,{}^t\!\mu].
\end{equation}
\subsection{}

Let us take an arc length parametrized curve 
$\gamma(s)=(x(s),y(s),\theta(s))$ in 
$G=\mathrm{SL}_2\mathbb{R}$. Then the anglar velocity 
$\varOmega$ is computed as
\[
\varOmega=Ae_1+Be_2+Ce_3,
\]
where 
\[
A=\dot{\theta}+\frac{\dot{x}}{2y}=\eta(\varOmega)=\cos\sigma,
\quad
B= \frac{\dot{x}}{2y}\cos(2\theta)
+\frac{\dot{y}}{2y}\sin(2\theta),
\quad 
C=-\frac{\dot{x}}{2y}\sin(2\theta)+
\frac{\dot{y}}{2y}\cos(2\theta).
\]
Since $\mathcal{I}$ is the transpose operation, we have
 \[
\mathcal{I}(e_1)=-e_1,
 \quad 
\mathcal{I}(e_2)=e_2,
 \quad 
\mathcal{I}(e_3)=e_3.
\]
Hence the momentum is 
 given by
 \[
\mu=-Ae_1+Be_2+Ce_3. 
 \]
Hence 
$[\mu,{}^t\!\mu]=[\mu,\varOmega]$ is computed as 
\[
[\mu,{}^t\!\mu]=
4A(Ce_2-Be_3).
\]
Hence the Euler-Arnold equation is 
\[
\dot{A}=0,
\quad 
\dot{B}=4AC,
\quad 
\dot{C}=-4AB.
\]
The first equation means the constancy of the contact angle.

Hence the Euler-Arnold equation reduces to the matrix-valued ODE:
\[
\frac{d}{ds}
\left(
\begin{array}{c}
B
\\
C
\end{array}
\right)
=-4\cos\sigma
\left(
\begin{array}{cc}
0 & -1\\
1 & 0
\end{array}
\right)
\left(
\begin{array}{c}
B
\\
C
\end{array}
\right).
\]
Thus the coefficient functions $B(s)$ and $C(s)$ are given 
explicitly by
\begin{align}\label{eq:BC}
\left(
\begin{array}{c}
B(s)
\\
C(s)
\end{array}
\right)
&=\exp_{K}(4s\cos\sigma\,e_1)
\left(
\begin{array}{c}
B(0)
\\
C(0)
\end{array}
\right)
\\
&=
\left(
\begin{array}{cc}
\cos(4s\cos\sigma) & \sin(4s\cos\sigma)\\
-\sin(4s\cos\sigma) & \cos(4s\cos\sigma)
\end{array}
\right)
\left(
\begin{array}{c}
B(0)
\\
C(0)
\end{array}
\right).
\nonumber
\end{align}
Split $\varOmega$ as 
$\varOmega=\varOmega_{\mathfrak k}+\varOmega_{\mathfrak m}$ along 
$\mathfrak{g}=\mathfrak{k}\oplus\mathfrak{m}$, we notice 
that $\varOmega_{\mathfrak k}=Ae_1=A\xi$,
\[
\exp_{K}(4s\cos\sigma\,e_1)=\exp_{K}(4s\varOmega_{\mathfrak k}).
\]
Let us demand the initial condition
\[
\gamma(0)=(x(0),y(0),\theta(0))=(0,1,0)
\]
and 
\[
\varOmega(0)=A(0)e_1+B(0)e_2+C(0)e_{3}=\cos\sigma\,e_1+
B_{0}e_2+C_0e_{3},
\]
where
\[
B_0=\frac{\dot{x}(0)}{2},
\quad 
C_0=\frac{\dot{y}(0)}{2},
\quad 
\cos\sigma=\dot{\theta}(0)+B_0.
\]
\begin{proposition}\label{prop:3.1}
Let $\gamma(s)$ be a geodesic of $\mathrm{SL}_2\mathbb{R}$ starting at 
the identity $\mathrm{Id}$ whose initial angular velocity
is $\varOmega(0)=\cos\sigma e_1+B_0e_2+C_0e_3$. Then 
the angular velocity $\varOmega(s)$ of $\gamma(s)$ is given by
\[
\varOmega(s)=\cos\sigma \,e_1+B(s)e_2+C(s)e_3,
\]
where $\sigma$ is a constant and 
\begin{equation}\label{eq:3.3}
\left\{
\begin{array}{l}
B(s)=B_0\,\cos(4s\cos\sigma)+C_0\sin(4s\cos\sigma),
\\
C(s)=C_0\,\cos(4s\cos\sigma)-B_0\sin(4s\cos\sigma).
\end{array}
\right.
\end{equation}
\end{proposition}
\subsection{}
We wish to solve the 
ODE $\dot{\gamma}(s)=\gamma(s)\varOmega(s)$ where 
$\varOmega(s)$ is the one given in Proposition 
\ref{prop:3.1}. For this purpose, let us consider curves of the form:
\[
\gamma(s)=\exp_{G}(sW)\exp_{K}(-sV),
\quad W\in \mathfrak{sl}_2\mathbb{R},
\>V\in\mathfrak{so}(2).
\]
Here $\exp_G$ and $\exp_K$ are the exponential maps
\[
\exp_{G}:\mathfrak{sl}_{2}\mathbb{R}\to \mathrm{SL}_2\mathbb{R},
\quad 
\exp_{K}:\mathfrak{so}(2)\to \mathrm{SO}(2), 
\quad 
\]
respectively. 
Express $W$ and $V$ as
\[
W=ae_1+be_2+ce_3
=\left(
\begin{array}{cc}
c & a+b\\
-a+b & -c
\end{array}
\right)
,\quad V=ue_1
=\left(
\begin{array}{cc}
0 & u\\
-u & 0
\end{array}
\right).
\]
By using the 
relations
\begin{align*}
\mathrm{Ad}(\exp_{K}(sV))e_1=&e_1,
\\
\mathrm{Ad}(\exp_{K}(sV))e_2=&\cos(2us)\,e_2+
\sin(2us)\,e_3,
\\
\mathrm{Ad}(\exp_{K}(sV))e_2=&-\sin(2us)\,e_2+
\cos(2us)\,e_3,
\end{align*}
the angular velocity 
is computed as
\[
\varOmega(s)=\gamma(s)^{-1}
\dot{\gamma}(s)=
\mathrm{Ad}(\exp_{K}(sV))W-V
=w_1(s)e_1+w_2(s)e_2+w_3(s)e_3,
\]
where
\begin{equation}\label{eq:3.4}
\left\{
\begin{array}{l}
 w_{1}(s)= a-u,
\\
 w_{2}(s)=
b\cos(2us)-c\,\sin(2us),
\\
w_{3}(s)=b\sin(2us)+c\,\cos(2us).
\end{array}
\right.
\end{equation}
Obviously $w_1$ is constant along $\gamma$. 
Thus $\gamma(s)$ is a geodesic 
if and only if 
\begin{align*}
& \dot{w}_2-4w_1w_3=
2(u-2a)\{b\,\cos(2us)+c\sin(2us)\}=0,
\\
&\dot{w}_3+4w_1w_2=
-2(u-2a)\{b\,\cos(2us)-c\sin(2us)\}=0.
\end{align*}
Thus we conclude that 
$\gamma(s)=\exp_{G}(sW)\exp_{K}(-sV)$ is a geodesic 
if and only if $u=2a$. It follows that
\[
\gamma(s)=\exp_{G}
\{s(ae_1+be_2+ce_3)\}\,
\exp_{K}\{s(-2a e_1)\}.
\]
Set 
\[
X=X_{\mathfrak k}+X_{\mathfrak m},
\quad 
X_{\mathfrak k}=-ae_1,
\quad 
X_{\mathfrak m}=be_2+ce_3.
\]
Then 
$\gamma(s)$ is rewritten as
\[
\gamma(s)=\exp_{G}\{s(-X_{\mathfrak{k}}+X_{\mathfrak m})
\}
\exp_{K}\{s(2X_{\mathfrak{k}})\}.
\]
Comparing the systems \eqref{eq:3.3} and \eqref{eq:3.4},
the uniqueness of geodesics implies 
$a=-\cos\sigma$ and $b=B_0$ and $c=C_0$ and  hence the 
initial velocity is 
\[
X=\cos\sigma e_1+B_0e_2+C_0e_3\in\mathfrak{sl}_{2}\mathbb{R}.
\] 
Thus we retrive the following fundamental result.
\begin{theorem}[\cite{DZ,Gordon,HI06}]
\label{thm:DZ}
The geodesic $\gamma_{X}(s)$ starting at the origin 
$\mathrm{Id}$ of $\mathrm{SL}_2\mathbb{R}$ with 
initial velocity $X=X_{\mathfrak k}+X_{\mathfrak m}
\in T_{\mathrm{Id}}\mathrm{SL}_{2}\mathbb{R}=\mathfrak{g}$ is given 
explicitly by
\[
\gamma_{X}(s)
=\exp_{G}\{s(-X_{\mathfrak{k}}+X_{\mathfrak{m}})\}
\exp_{K}\{s(2X_{\mathfrak{k}})\}.
\]
\end{theorem}
Here we give homogeneous geometric interpretation for this 
fundamental theorem.

The exponential map 
$\exp_{G\times K}:\mathfrak{sl}_{2}\mathbb{R}\oplus\mathfrak{so}(2)
\to \mathrm{SL}_2\mathbb{R}\times \mathrm{SO}(2)$ is given by
\[
\exp_{G\times K}(X,Y)=(\exp_{G}X, \exp_{K}Y),\quad (X,Y)\in \mathfrak{sl}_{2}\mathbb{R}\oplus\mathfrak{so}(2).
\]
Take a vector $(X,Y)\in\mathfrak{sl}_{2}\mathbb{R}\oplus\mathfrak{so}(2)$, 
then the orbit of $\mathrm{Id}$ under the 
one-parameter subgroup 
$\{ \exp_{G\times K} \{s(X,Y)\} \,\}_{s\in\mathbb{R}
}\subset G\times K$ is given by
\[
\exp_{G\times K}\{s(X,Y)\}\cdot \mathrm{Id}
=\exp_{G}(sX)\cdot\mathrm{Id}\cdot 
\left(\exp_{K}(sY)\,\right)^{-1}
=\exp_{G}(sX)\exp_{K}(-sY).
\]
Since $\mathrm{SL}_2\mathbb{R}
=(\mathrm{SL}_2\mathbb{R}\times\mathrm{SO}(2))/\mathrm{SO}(2)$ 
is naturally reductive, 
every geodesic starting at the origin 
$\mathrm{Id}$ is an orbit of $\mathrm{Id}$ under the 
one-parameter subgroup 
of $G\times K$ with initial velocity 
$X\in T_{\mathrm{Id}}\mathrm{SL}_{2}\mathbb{R}=\mathfrak{g}$. 
Under the identification of 
$T_{\mathrm{Id}}\mathrm{SL}_{2}\mathbb{R}$ with 
$\mathfrak{p}$, the initial velocity 
$X=X_{\mathfrak{k}}+X_{\mathfrak{m}}$ is 
identified with 
$(X_{\mathfrak{m}}-X_{\mathfrak k},-2X_{\mathfrak k})$.
Thus the geodesic 
\[
\gamma_{X}(s)
=\exp_{G}\{s(-X_{\mathfrak{k}}+X_{\mathfrak{m}})\}
\exp_{K}\{s(2X_{\mathfrak{k}})\}.
\] 
is nothing but the orbit $\exp_{G\times K}(sX)\,\mathrm{Id}$.

According to \cite{BVY}, the fact that the modular 
$3$-fold $\mathrm{PSL}_{2}\mathbb{R}/\mathrm{PSL}_{2}\mathbb{Z}$ 
is topologically equivalent to the complement of the 
trefoil $\mathcal{K}$ in the $3$-sphere was first observed by Quillen 
(see Milnor \cite{Milnor}). 
Bolsinov, Veselov and Ye \cite{BVY} proved that
the periodic geodesics on the modular $3$-fold 
$\mathrm{SL}_{2}\mathbb{R}/\mathrm{SL}_{2}\mathbb{Z}$ 
with sufficiently large values of $\mathcal{C}=\kappa^2/16$
represent trefoil cable knots in 
$\mathbb{S}^3\smallsetminus \mathcal{K}$, where $\kappa$ is the 
geodesic curvature of the projected curve in the hyperbolic surface. 
Any trefoil cable knot can be described in this way.


\section{Homogeneous magnetic trajectories in $\mathrm{SL}_{2}\mathbb{R}$}

\subsection{The magnetized Euler-Arnold equation}
Now we magnetize the Euler-Arnold equation. 
Since $\varphi$ is left invariant, 
\[
\varphi\dot{\gamma}=\gamma \varphi \varOmega.
\]
Hence 
\[
\varphi\varOmega
=\mathcal{I}^{-1}\mathcal{I}\varphi\varOmega
=\mathcal{I}^{-1}(\mathcal{I}\varphi\mathcal{I}\mu).
\]
One can confirm that $\mathcal{I}\varphi\mathcal{I}=\varphi$.
Hence the Lorentz equation is 
rewritten as the following 
\emph{magnetized Euler-Arnold equation}:
\begin{equation}\label{eq:EA-mag}\
\dot{\mu}-[\mu,{}^t\!\mu]=q\varphi\mu.
\end{equation}
Express $\mu=-Ae_1+Be_2+Ce_3$ as before, then 
the left hand side of the magnetized Euler-Arnold equation is 
\[-\dot{A}e_1+(\dot{B}-4AC)e_2+
(\dot{C}+4AB)e_3.
\]
On the other hand,
\[
\varphi \mu=-Ce_2+Be_3.
\]
Hence the magnetized Euler-Arnold equation is the system
\begin{align*}
&\dot{A}=0,\\
&\dot{B}=(4A-q)C,
\\
&\dot{C}=-(4A-q)B.
\end{align*}
One can check that this system 
coincides the system 
obtained in \cite[p.~2181]{IM1}. 
The first equation implies the 
constancy of the contact angle $A=\cos\sigma$. 
Under the initial condition $B(0)=B_0$ and 
$C(0)=C_0$, 
we obtain
\begin{align}\label{eq:4BC}
\left(
\begin{array}{c}
B(s)
\\
C(s)
\end{array}
\right)
&=\exp_{K}(s(4\cos\sigma-q)\,e_1)
\left(
\begin{array}{c}
B_0
\\
C_0
\end{array}
\right)
\\
&=
\left(
\begin{array}{cc}
\cos((4\cos\sigma-q)s) & \sin((4\cos\sigma-q)s)\\
-\sin((4\cos\sigma-q)s) & \cos((4\cos\sigma-q)s)
\end{array}
\right)
\left(
\begin{array}{c}
B_0
\\
C_0
\end{array}
\right).
\nonumber
\end{align}
\subsection{}
To solve the ODE system \eqref{eq:4BC}, 
here we determine contact magnetic curves 
of the form
\[
\gamma(s)=\exp_{G}(sW)\exp_{K}(-sV),
\quad (W,V)\in\mathfrak{sl}_2\mathbb{R}\oplus\mathfrak{so}(2).
\]
Note that when $(W,V)\in\mathfrak{p}$, then $\gamma(s)$ is a geodesic. 
Express $W$ and $V$ as
\[
W=ae_1+be_2+ce_3
=\left(
\begin{array}{cc}
c & a+b\\
-a+b & -c
\end{array}
\right)
,\quad V=ue_1
=\left(
\begin{array}{cc}
0 & u\\
-u & 0
\end{array}
\right).
\]
Use the same notations as before, that is 
$\Omega(s)=\gamma(s)^{-1}\dot{\gamma}(s):=
w_1(s)e_1+w_2(s)e_2+w_3(s)e_3.
$
Thus $\gamma(s)$ is a contact magnetic trajectory 
if and only if 
\begin{align*}
& \dot{w}_2-4w_1w_3+qw_3=~ (2u-4a+q)(b\sin(2us)+c\,\cos(2us))=0,
\\
& \dot{w}_3+4w_1w_2-qw_2=-(2u-4a+q)(b\cos(2us)-c\,\sin(2us))=0.
\end{align*}
Thus we conclude that 
$\gamma(s)=\exp_{G}(sW)\exp_{K}(-sV)$ is a contact magnetic trajectory   
when and only when $u=2a-q/2$. Namely
\[
\gamma(s)=\exp_{G}
\{s
(ae_1+be_2+ce_3)\}\,
\exp_{K}
\left\{
-s(2a-\tfrac{q}{2})e_1
\right\}.
\]
Set 
\[
X=X_{\mathfrak k}+X_{\mathfrak m},
\quad 
X_{\mathfrak k}=-ae_1,
\quad 
X_{\mathfrak m}=be_2+ce_3.
\]
Then 
$\gamma(s)$ is rewritten as
\begin{align*}
\gamma(s)=&\exp_{G}\{s(-X_{\mathfrak{k}}+X_{\mathfrak m})
\}
\exp_{K}\{s(2X_{\mathfrak{k}}+\tfrac{q}{2}\xi)\}
\\
=&
\exp_{G}\{s(-X_{\mathfrak{k}}+X_{\mathfrak m})
\}
\exp_{K}\{s(2X_{\mathfrak{k}})\}
\exp_{K}\{s(\tfrac{q}{2}\xi)\}
\\
=&\gamma_{X}(s)\,\exp_{K}\{s(\tfrac{q}{2}\xi)\}.
\end{align*}
We may state the following result.
\begin{theorem}
\label{MagHom_THM}
Every contact magnetic curve $\gamma$ of $\mathrm{SL}_2\mathbb{R}$ is homogeneous.
\end{theorem}
More precisely, we may explicitly describe contact magnetic curves in $\mathrm{SL}_2\mathbb{R}$.
\begin{theorem}
\label{MagTrajProd_THM}
Every contact magnetic curve $\gamma$ starting at the origin 
$\mathrm{Id}$ of $\mathrm{SL}_2\mathbb{R}$ with 
initial velocity $X\in T_{\mathrm{Id}}\mathrm{SL}_{2}\mathbb{R}=\mathfrak{g}$ 
and with charge $q$ is the product of the
homogeneous geodesic $\gamma_{X}(s)$ and the
\emph{charged Reeb flow} $\exp_{K}\{s(\tfrac{q}{2}\xi)\}$, namely
\[
\gamma(s)=\gamma_{X}(s)\,\exp_{K}\{s(\tfrac{q}{2}\xi)\}.
\]
\end{theorem}
\begin{remark}{ \rm
Recall that contact magnetic curves in $\mathrm{SL}_2\mathbb{R}$ are slant 
(see \textit{e.g.}, \cite{DIMN1,IM1}).
Theorem~$\ref{MagTrajProd_THM}$ is another justification of this fact. 
Magazev, Shirokov and Yurevich studied 
integrability of magnetic trajectory on 
Lie groups of higher dimension \cite{MSY}.
}
\end{remark}


\appendix

\section{Remarks on homogeneous geodesics
of the form $\exp_G(sX)$}
In this appendix we investigate homogeneous geodesics
of the form $\exp_G(sX)$. 

\begin{lemma}
For $X=X_{\mathfrak{l}}+X_{\mathfrak{m}}\in\mathfrak{g}={\mathfrak{sl}}_2\mathbb{R}$, the 
following properties are mutually equivalent{\rm:}
\begin{itemize}
\item $X_{\mathfrak{k}}=0$ or $X_{\mathfrak{m}}=0$.
\item $\mathsf{U}(X,X)=0$.
\end{itemize}
\end{lemma}
\begin{proof}
For $X\in\mathfrak{g}$, its $\mathfrak{k}$-component and 
$\mathfrak{m}$-component are
\[
X_{\mathfrak k}=\frac{1}{2}\left(
\begin{array}{cc}
0 & X_{12}-X_{21}
\\
X_{21}-X_{12} & 0
\end{array}
\right), 
\quad 
X_{\mathfrak{m}}
=
\frac{1}{2}\left(
\begin{array}{cc}
2X_{11} & X_{12}+X_{21}
\\
X_{21}+X_{12} & -2X_{11}
\end{array}
\right),
\]
\[
-X_{\mathfrak k}+X_{\mathfrak m}
=
\frac{1}{2}\left(
\begin{array}{cc}
X_{11} & X_{21}
\\
X_{12} & -X_{11}
\end{array}
\right)
={}^t\!X.
\]
\[
\exp_{K}\{s(2X_{\mathfrak k})\}
=
\left(
\begin{array}{cc}
\cos\{s(X_{21}-X_{12})\} &-\sin\{s(X_{21}-X_{12})\\
\sin\{s(X_{21}-X_{12}) & \cos\{s(X_{21}-X_{12})\}
\end{array}
\right)
\]
Thus $X_{\mathfrak{k}}=0$ if and only if $X_{12}=X_{21}$. 
On the other hand we know that 
$\mathsf{U}(X,X)=0$ if and only if 
$[X_{\mathfrak{k}},X_{\mathfrak{m}}]=0$.
\[
X_{\mathfrak{k}}=\frac{1}{2}(X_{21}-X_{12})(F-E),
\quad 
X_{\mathfrak{m}}=\frac{1}{2}(X_{21}+X_{12})(F+E)+X_{3}H
\]
Hence
\[
[X_{\mathfrak{k}},X_{\mathfrak{m}}]
=-\frac{1}{2}(X_{21}^2-X_{12}^2)H
+(X_{21}-X_{12})X_{11}(E+F).
\]
Thus $[X_{\mathfrak{k}},X_{\mathfrak{m}}]=0$ 
if and only if 
${}^t\!X=X$ or $X_{11}=X_{12}+X_{21}=0$. 
In the first case, $X_\mathfrak{k}=0$. 
In the latter case, $X_\mathfrak{m}=0$.
\end{proof}
We obtain the following fact.
\begin{proposition}\label{prop:4.2}
A curve $\exp_G(sX)$ with $X\in\mathfrak{sl}_2\mathbb{R}$ is a
geodesic starting at the origin if and only if 
\begin{itemize}
\item either $X$ is a symmetric matrix,
\item or $X\in\mathfrak{k}$. In this case 
$\exp_{G}(sX)=\exp_{K}(sX)$. 
\end{itemize}
\end{proposition}
\begin{remark}{\rm In \cite{Ko}, Kowalski 
gave the following comment to \cite{IKOS3}:
\begin{quotation}
One remark to section 3.5: 
it is correct that not all 1-parameter subgroups of 
$(SL_{2}\mathbb{R},g_{\lambda})$ are geodesics and vice versa. 
Yet, because every such manifold is naturally reductive, 
all its geodesics are orbits of 1-parameter groups of isometries.
\end{quotation}
Proposition \ref{prop:4.2} gives a more clear explanation 
for the comment by Kowalski. 
The $1$-parameter subgroups of $\mathrm{SL}_2\mathbb{R}$ 
mentioned in Kowalski's review 
mean $\{\exp_{G}(sX)\}_{s\in\mathbb{R}}$ for some 
$X\in\mathfrak{sl}_{2}\mathbb{R}$. 

On the other hand, since $\mathrm{SL}_2\mathbb{R}$ is 
naturally reductive with respect to the 
reductive decomposition 
$\mathfrak{g}\oplus\mathfrak{k}=\Delta\mathfrak{k}+\mathfrak{p}$, 
every geodesic starting at the origin $\mathrm{Id}$ is 
homogeneous and has the form
\[
\gamma(s)=\exp_{G\times K}\{s(X,Y)\}\,\mathrm{Id}
=\exp_{G}(sX)\exp_K(-sY)
\]
for some $(X,Y)\in\mathfrak{p}$.
As we quoted in Theorem \ref{thm:DZ}, 
the precise form of a homogeneous geodesic is
\[
\gamma(s)=
\exp_{G\times K}(sX)\,\mathrm{Id}
=\exp_{G}\{
s(-X_{\mathfrak{k}}+X_{\mathfrak m})\}
\exp_K\{
s(2X_{\mathfrak{k}})
\},\quad X\in\mathfrak{g}.
\]
However a curve of the form $\exp_{G}(sX)$ is a homogeneous 
geodesic if and only if either $X=X_{\mathfrak{m}}$ or 
$X=X_{\mathfrak{k}}$.
}
\end{remark}

Set $E_1=\sqrt{2}E$, $E_2=\sqrt{2}F$ and $E_3=H$ such that they are orthonormal with respect to
$\langle~,~\rangle$.

Let us express $X\in\mathfrak{g}$ as 
\[
X=aE_1+bE_2+cE_3=
\left(
\begin{array}{cc}
c& a\sqrt{2}\\
 b\sqrt{2} & -c
\end{array}
\right).
\]
Then
$X_{21}-X_{12}=\sqrt{2}(b-a)$ and $X_{11}=c$,
$X_{12}=\sqrt{2}a$ and $X_{21}=\sqrt{2}b$.

\[
X_{\mathfrak k}=\frac{1}{2}\left(
\begin{array}{cc}
0 & -\sqrt{2}(b-a)
\\
\sqrt{2}(b-a) & 0
\end{array}
\right), 
\quad 
X_{\mathfrak{m}}
=
\frac{1}{2}\left(
\begin{array}{cc}
2c & \sqrt{2}(b+a)
\\
\sqrt{2}(b+a) & -2c
\end{array}
\right),
\]
\[
-X_{\mathfrak k}+X_{\mathfrak m}
=
\frac{1}{2}\left(
\begin{array}{cc}
c & \sqrt{2}b
\\
\sqrt{2}a & -c
\end{array}
\right)
={}^t\!X.
\]
Let us describe these homogeneous geodesics explicitly.

Firstly we assume that $X=X_\mathfrak{m}$, i.e. $X={}^t\!X$ (that is $b=a$). 
Then
$\det X=-c^2-2b^2\leq 0$. 
We cannot have $\det X=0$, otherwise we obtain $X=0$, which is a contradiction.
 Set $\delta=\sqrt{c^2+2b^2}>0$. We get
 \[
 \exp_{G}(sX)
 =\left(
 \begin{array}{cc}
 \cosh(\delta s)+(c/\delta)\sinh(\delta s)
 &(b\sqrt{2}/\delta)\sinh(\delta s)\\
(b\sqrt{2}/\delta)\sinh(\delta s)
& 
 \cosh(\delta s)-(c/\delta)\sinh(\delta s)
 \end{array}
 \right)
 \]
 We compute now the coordinates $(x,y,\theta)$ of $\exp(sX)$ 
 by using Proposition \ref{Iwasawa}:
\begin{align*}
x(s)=&\frac{b\sqrt{2}\sinh(2\delta s)}
{\delta\cosh(2\delta s)-c\sinh(2\delta s)},
\\
y(s)=&\frac{\delta}
{\delta\cosh(2\delta s)-c\sinh(2\delta s)},,
\\
 e^{i\theta(s)}=&
 \frac{\cosh(\delta s)-\frac{c}{\delta}\sinh(\delta s)-i~\frac{b\sqrt{2}}{\delta}\sinh(\delta s)}
 {\sqrt{\cosh(2\delta s)-\frac{c}{\delta}\sinh(2\delta s)}}.
\end{align*}
 The projected curve $(x(s),y(s))$ is a geodesic in $\mathbb{H}^2(-4)$ given as follows:

If $b\not=0$, then 
\[
\left(x-\frac{c}{b\sqrt{2}}\right)^2
+y^2=\frac{\delta^2}{2b^2},\quad y>0.
\]
If $b=0$, 
\[
x=0, \quad y>0.
\]

Secondly we assume that $X=X_{\mathfrak{k}}$ (that is $c=0$, $a=-b$). Then
\[
X=\sqrt{2}b\left(
\begin{array}{cc}
0 & -1\\
1 & 0
\end{array}
\right).
\]
Hence
\[
\exp_{G}(sX)=\exp_{K}(sX)=
\left(
\begin{array}{cc}
\cos(\sqrt{2}bs) & -\sin(\sqrt{2}bs) \\
\sin(\sqrt{2}bs) & \cos(\sqrt{2}bs)
\end{array}
\right),
\]
that implies
\[ x(s)=0,\quad y(s)=1, \quad
\theta(s)=-\sqrt{2}bs.
\]

\begin{theorem}
All homogeneous geodesics 
of the form $\exp_{G}(sX)$ with 
$X=aE_1+bE_2+cE_3\in\mathfrak{g}$ belong to the following list
\begin{enumerate}
\item When $a=b\neq0$ and $\delta=\sqrt{2b^2+c^2}$, then 
\begin{align*}
x(s)=&\frac{b\sqrt{2}\sinh(2\delta s)}
{\delta\cosh(2\delta s)-c\sinh(2\delta s)},
\\
y(s)=&\frac{\delta}
{\delta\cosh(2\delta s)-c\sinh(2\delta s)},,
\\
 e^{i\theta(s)}=&
 \frac{\cosh(\delta s)-\frac{c}{\delta}\sinh(\delta s)-i~\frac{b\sqrt{2}}{\delta}\sinh(\delta s)}
 {\sqrt{\cosh(2\delta s)-\frac{c}{\delta}\sinh(2\delta s)}}.
\end{align*}
The projected curve is a geodesic in $\mathbb{H}^2(-4)$.
\item When $a=b=0$ then
$$
x(s)=0,\quad y(s)=e^{2cs} , \quad \theta(s)=0.
$$
The projected curve is the half-line geodesic $x=0$, $y>0$.
\item When $a=-b\not=0$ and $c=0$, then
\[
x(s)=0,\quad y(s)=1,\quad \theta(s)=-\sqrt{2}bs.
\]
This is a vertical geodesic with respect to the hyperbolic Hopf fibration.
The spur of this geodesic is the fiber over $(0,1)\in
\mathbb{H}^{2}(-4)$.
\end{enumerate}
\end{theorem}

\section{Sasakian space forms of constant holomorphic sectional curvature $c<-3$.}

Let $\mathscr{M}^{3}(c)=(\mathscr{M}^{3}(c),\bar{\varphi},\bar{\xi},\bar{\eta},\bar{g})$ be a $3$-dimensional 
Sasakian space form of constant 
holomorphic sectional curvature $c<-3$. 
Let us perform the following 
deformation of the structure tensor fields 
(Tanno's $\mathcal{D}$-homothetic deformation):
\begin{equation}
g:=-\frac{c+3}{4}\bar{g}+\frac{(c+3)(c+7)}{16}\bar{\eta}\otimes\bar{\eta}
\quad 
\varphi:=\bar{\varphi},
\quad
\xi:=-\frac{4}{c+3}\bar{\xi},
\quad
\eta:=-\frac{c+3}{4}\bar{\eta}.
\end{equation}
Then the new structure $(\varphi,\xi,\eta,g)$ is a 
Sasakian structure of constant holomorphic sectional 
curvature $-7$. Hence $\mathscr{M}^{3}(c)$ is locally 
isomorphic to $\mathrm{SL}_{2}\mathbb{R}$ equipped with a  
Sasakian structure of constant holomorphic sectional curvature $-7$. 

Conversely, every $3$-dimensional 
Sasakian space form of constant 
holomorphic sectional curvature $c<-3$ is locally 
obtained as the $\mathcal{D}$-homothetic deformation of 
$\mathrm{SL}_{2}\mathbb{R}$ equipped with a  
Sasakian structure of constant holomorphic sectional curvature $-7$.

Thus we may identify $\mathcal{M}^{3}(c)$ with 
the real special linear group $\mathrm{SL}_{2}\mathbb{R}$ 
equipped with a left invariant Sasakian structure 
$(\bar{\varphi},\bar{\xi},\bar{\eta},\bar{g})$ defined by
\[
\bar{g}:=-\frac{4}{c+3}\,g+\frac{4(c+7)}{(c+3)^2}\,\eta\otimes\eta,
\quad 
\bar{\varphi}:=\varphi:=\bar{\varphi},
\quad
\bar{\xi}:=-\frac{c+3}{4}\bar{\xi},
\quad
\bar{\eta}:=-\frac{4}{c+3}\bar{\eta}.
\]
The Sasakian space form $(\mathrm{SL}_2\mathbb{R},\bar{\varphi},\bar{\xi},\bar{\eta},\bar{g})$ is 
a naturally reductive homogeneous space $(G\times K)/\Delta K$ with reductive decomposition
$\mathfrak{g}\oplus\mathfrak{k}=\Delta\mathfrak{k}\oplus\mathfrak{p}_{c}$, where
\[
\mathfrak{p}_{c}
=\left\{\left.
\left(
V+W,-\frac{c-1}{4}W
\right)
\>\right|\>V\in\mathfrak{m},\>
W\in\mathfrak{k}
\right\}.
\]
Theorem~\ref{thm:DZ} and Theorem~\ref{MagTrajProd_THM} may be modified as follows:

\begin{theorem}
The geodesic $\gamma_{X}(s)$ starting at the origin $\mathrm{Id}$ 
of the $\mathrm{SL}_{2}\mathbb{R}$ of holomorphic 
sectional curvature $c<-3$ with initial velocity 
$X=X_{\mathfrak{k}}+X_{\mathfrak{m}}\in 
\mathfrak{g}$ is given 
explicitly by
\begin{equation}\label{eq:homgeo}
\gamma_{X}(s)=\exp_{G\times K}(sX)
=\exp_{G}
\left\{s
\left(
X_{\mathfrak{m}}+
\frac{4}{c+3}X_{\mathfrak{k}}
\right)
\right\}
\,
\exp_{K}
\left\{
s\left(
\frac{c-1}{c+3}X_{\mathfrak k}
\right)
\right\}.
\end{equation}
\end{theorem}

\begin{theorem}
Every contact magnetic curve $\gamma$ starting at the origin $\mathrm{Id}$ 
of the $\mathrm{SL}_{2}\mathbb{R}$ of holomorphic 
sectional curvature $c<-3$ with initial 
velocity $X\in T_{\mathrm{Id}}\mathrm{SL}_{2}\mathbb{R}=\mathfrak{g}$ 
and with charge $q$ is the product of the homogeneous 
geodesic $\gamma_{X}(s)$ and the charged Reeb flow $\exp_{K}\{s(\frac{q}{2}\bar{\xi})\}$. 
\end{theorem}

{\bf Acknowledgement.}
The first named author was partially supported by JSPS KAKENHI JP19K03461 and JP23K03081.
The second named author was partially supported by Romanian Ministry of Research, 
Innovation and Digitization, within Program 1 – Development of the national 
RD system, Subprogram 1.2 – Institutional Performance – RDI excellence funding projects, 
Contract no.11PFE/30.12.2021.



\begin{thebibliography}{99}

\bibitem{Arnold61}
V.~I.~Arnol'd, 
\emph{Some remarks on flows of line elements and frames}, 
Dokl. Akad. Nauk. SSSR \textbf{138} (1961), 255--257. 
English translation: Sov. Math. Dokl. \textbf{2} (1961), 562--564. 

\bibitem{Arnold}
V.~I.~Arnold, 
\emph{Mathematical Methods of Classical Mechanics}, Second Edition, 
Grad. Texts in Math. \textbf{60}, Springer-Verlag, 1989.

\bibitem{Av}
A.~Arvanitoyeorgos,
\emph{Homogeneous manifolds whose geodesics are orbits. Recent results and some open problems},
Irish Math. Soc. Bull. {\bf 79} (2017) 5--29 (see also \texttt{arXiv:1706.09618}).

\bibitem{BaTa}
A.~Bahri, I.~A.~Taimanov, 
\emph{Periodic orbits in magnetic fields and Ricci curvature of Lagrangian systems}, 
Trans. Amer. Math. Soc. \textbf{350} (1998), no.~7, 2697--2717. 

\bibitem{BJ03}
A.~V.~Bolsinov, 
B.~Jovanovi{\'c}, 
\emph{Noncommutative integrability, moment map and geodesic flows}, 
Annals of Global Analysis and Geometry 
\textbf{23} (2003), 305--322.

\bibitem{BJ}
A.~V.~Bolsinov, 
B.~Jovanovi{\'c}, 
\emph{Magnetic flows on homogeneous spaces}, 
Comment. Math. Helv. \textbf{83} (2008), no.~3, 679--700.



\bibitem{BT}
A.~V.~Bolsinov, I.~A.~Taimanov, 
\emph{Integrable geodesic flows with positive topological entropy}, 
Invent. Math. \textbf{140} (2000), no.~3, 639--650.

\bibitem{BVY}
A.~V.~Bolsinov, 
A.~P.~Veselov, 
Y.~Ye, 
\emph{Chaos and integrability in $\mathrm{SL}(2,\mathbb{R})$-geometry}, 
Russian Math. Surv. \textbf{76} (2021), no.~4, 557--586.

	
\bibitem{CFG} J.~L.~Cabrerizo,  M.~Fern\'andez, and J.~S.~G\'omez,
    \emph{The contact magnetic flow in $3D$ Sasakian manifolds},
    J. Phys. A: Math. Theor. {\bf 42} (2009), 19, 195201.
		


\bibitem{DZ}
J.~E.~D'Atri, W.~Ziller, 
\emph{Naturally Reductive Metrics and Einstein Metrics on Compact Lie Groups}, 
Mem. Amer. Math. Soc. \textbf{215} (1979).


\bibitem{DIMN1} S.~L.~Dru\c t\u a-Romaniuc, J.~Inoguchi, M.~I.~Munteanu, and A.~I.~Nistor,
	\emph{Magnetic curves in Sasakian manifolds}, J. Nonlinear Math. Phys. 
	{\bf 22} (2015), no.~3, 428--447.


\bibitem{Gordon}
C.~S.~Gordon, \emph{Naturally reductive homogeneous Riemannian manifolds}, 
Canadian J. Math. \textbf{37} (1985), no.~3, 467--487. 

\bibitem{HI06}
S.~Halverscheid, 
A.~Iannuzzi,
\emph{On naturally reductive left-invariant metrics of 
$\mathrm{SL}(2,\mathbb{R})$}, 
Ann. Sc. Norm. Super. Pisa Cl. Sci. (5) \textbf{5} (2006), no.~2, 
171--187. 

\bibitem{HMR}
D.~D.~Holm, 
J.~E.~Marsden, 
T.~S.~Ratiu,  
\emph{The Euler–Poincar{\'e} equations and semidirect products with applications to 
continuum theories}, 
Advances in Math. \textbf{137} (1998), no.~1, 1--81.

\bibitem{I2004}
J.~Inoguchi, 
\emph{Invariant minimal surfaces in 
the real special linear group of 
degree $2$},
Italian J. Pure Appl. Math. 
\textbf{16} (2004), 61--80. 
\texttt{arXiv:0910.3104}


\bibitem{IKOS3}
J.~Inoguchi, T.~Kumamoto,
N.~Ohsugi and
Y.~Suyama,
\emph{Differential geometry of curves and surfaces in 
$3$-dimensional homogeneous spaces 
$\mathrm{I}\!\mathrm{I}\!\mathrm{I}$},
Fukuoka Univ. Sci. Rep.
{\bf 30} (2000),
131--160.



\bibitem{IM1} J.~Inoguchi and M.~I.~Munteanu, 
	\emph{Magnetic curves in the real special linear group}, 
	Adv. Theor. Math. Phys. 
	{\bf 23} (2019), no.~8, 2161--2205.


\bibitem{IM2}
J.~Inoguchi and M.~I.~Munteanu, 
\emph{Slant curves and magnetic curves}, in: 
Contact Geometry of Slant Submanifolds, 
Springer Nature Singapore, Singapore, 
 2022, pp.~199--259.
 
\bibitem{IM3} 
J.~Inoguchi and M.~I.~Munteanu, 
\emph{Magnetic Jacobi fields in $3$-dimensional Sasakian space forms}, 
The Journal of Geometric Analysis \textbf{32} (2022), no.~33, 
Article number 96. 

\bibitem{Jov}
B.~Jovanovi{\'c}, 
Geodesic flows on Riemannian g.~o.~spaces, 
\emph{Regul. Chaot. Dyn.} \textbf{16} (2011), 504--513.

\bibitem{Ko}
O.~Kowalski, Zbl 0991.53031, a review to \cite{IKOS3}.

\bibitem{KV}
O.~Kowalski, L.~Vanhecke, 
\emph{Riemannian manifolds with homogeneous geodesics}, 
Boll. Un. Math. Ital. B (7) \textbf{5} (1991), no.~1, 189--246.

\bibitem{MSY}
A.~A.~Magaz{\"e}v, 
I.~V.~Shirokov, 
Yu.~A.~Yurevich, 
\emph{Integrable magnetic geodesic flows on Lie groups}, 
Teoret. Mat. Fiz. \textbf{156} (2008), no.~2, 189--206. 
English translation: 
Theoret. and Math. Phys. \textbf{156} (2008), no.~2, 1127--1141.

\bibitem{Milnor}
J.~Milnor, 
\emph{Introduction to Algebraic $K$-Theory}, 
Ann. of Math. Stud., vol. 72, Princeton Univ. Press, Princeton, NJ; Univ. of Tokyo Press, Tokyo 1971.

\bibitem{MPMM}
K.~Modin, 
M.~Perlmutter, 
S.~ Marsland, 
R.~McLachlan, 
\emph{On Euler–Arnold equations and totally geodesic subgroups}, 
J. Geom. Phys. \textbf{61} (2011), no.~8, 1446--1461.


\bibitem{Thurston}
W.~M.~Thurston, 
\emph{Three-dimensional Geometry and Topology} I, 
Princeton Math. Series., vol.~35 (S.~Levy ed.), 1997.

\end{thebibliography}
\end{document}